\documentclass[oneside,english]{amsart}
\usepackage[T1]{fontenc}
\usepackage[latin9]{inputenc}
\usepackage{geometry}
\usepackage{color}
\usepackage{babel}
\usepackage{verbatim}
\usepackage{units}
\usepackage{amsthm}
\usepackage{amssymb}
\usepackage{esint}
\usepackage[unicode=true,pdfusetitle,
 bookmarks=true,bookmarksnumbered=false,bookmarksopen=false,
 breaklinks=false,pdfborder={0 0 0},backref=false,colorlinks=true]
 {hyperref}

\makeatletter
\numberwithin{equation}{section}
\numberwithin{figure}{section}
  \theoremstyle{remark}
  \newtheorem{rem}{\protect\remarkname}
\theoremstyle{plain}
\newtheorem{thm}{\protect\theoremname}
  \theoremstyle{plain}
  \newtheorem{cor}{\protect\corollaryname}
  \theoremstyle{plain}
  \newtheorem{lem}{\protect\lemmaname}
  \theoremstyle{plain}
  \newtheorem{prop}{\protect\propositionname}

\IfFileExists{lmodern.sty}{\usepackage{lmodern}}{}

\makeatother

  \providecommand{\lemmaname}{Lemma}
  \providecommand{\propositionname}{Proposition}
  \providecommand{\remarkname}{Remark}
\providecommand{\corollaryname}{Corollary}
\providecommand{\theoremname}{Theorem}

\begin{document}

\title{Aging Uncoupled Continuous Time Random Walk Limits}

\author{Ofer Busani}

\maketitle
\begin{minipage}[t]{1\columnwidth}%
Department of mathematics \\
Bar-Ilan University, Ramat Gan, Israel \\%
\end{minipage}

\begin{abstract}
Aging is a prevalent phenomenon in physics, chemistry and many other
fields. In this paper we consider the aging process of uncoupled Continuous
Time Random Walk Limits (CTRWLs) which are Levy processes time changed
by the inverse stable subordinator of index $0<\alpha<1$. We apply
a recent method developed by Meerscheart and Straka of finding the
finite dimensional distributions of CTRWL, to obtaining the aging
process\textquoteright s finite dimensional distributions, self-similarity-like
property, asymptotic behavior and its Fractional Fokker-Planck equation(FFPE). 
\end{abstract}

\global\long\def\integral#1#2{{\displaystyle \intop_{#1}^{#2}}}

\section{Introduction\label{sec:Introduction}}

Continuous time random walks(CTRW) are widely used in physics and
mathematical finance to model a random walk for which the waiting
times between jumps are random which in many cases better describes
phenomena in these fields. CTRWLs are used to model anomalous diffusion,
where the squared averaged distance of the process from the origin
is no longer proportional to the time index $t$. A related concept
and widely studied (\cite{struik1978physical,Schulz2013a}) in statistical
physics, is aging. Suppose the CTRW $X_{t}$ starts at $t=0$ and
evolves until time $t_{0}>0$ when we then start to measure it. One
can consider the varying dynamics of the new process $X_{t}^{t_{0}}=X_{t+t_{0}}-X_{t_{0}}$
as $t_{0}$ varies and the process ages. In \cite{monthus1996models}
Monthus and Bouchaud studied a CTRW with aging properties. In \cite{Barkai2003}
Barkai and Cheng considered the Aging Continuous Time Random Walk
(ACTRW) which is an uncoupled CTRW with iid power law waiting times,
that started at $t=0$ and is observed at $t=t_{0}$. They found the
one dimensional distribution of the process $X_{t}^{t_{0}}$ which
they referred to as the ACTRW, for $t_{0}$ and $t$ large. In \cite{Barkai2003a},
Barkai found the Fractional Fokker-Planck Equation (FFPE) for the
unnormalized pdf of the process $X_{t}^{t{}_{0}}$ for  $t_{0}$ and
$t$ large.\\
In this paper we wish to give analogous results to the ones given
in \cite{Barkai2003,Barkai2003a} as well as new ones for a large
class of CTRWLs which hopefully will lay the foundation for further
study of their aging. We consider the class that consists of all processes
of the form $Y_{t}=A_{E_{t}}$ where $A_{t}$ is a Levy process that
is time changed by the inverse of an independent stable subordinator
of index $0<\alpha<1$ ; we denote this class by $\mathcal{S}$. We
denote the aging process by $Y_{t}^{t_{0}}=Y_{t+t_{0}}-Y_{t_{0}}=A_{E_{t+t_{0}}}-A_{E_{t_{0}}}$
(note that $Y_{t}^{0}=Y_{t}$). Section \ref{sec:Finite-dimensional-distribution}
is devoted to a brief review of the theory and method introduced by
Meerschaert and Straka in \cite{Meerschaert} and \cite{Meerschaert2012}
upon which we base our results. In Section \ref{sec:aging} we give
the main result of this paper, that the finite dimensional distributions
of the process $Y_{t}^{t_{0}}$ can be obtained by a convolution in
time of the finite dimensional distributions of $Y_{t}$ and a generalized
beta prime distribution. The self-similarity-like property of the
process $Y_{t}^{t_{0}}$ is obtained in Section \ref{sec:Aging-self-similarity}.
In Section \ref{sec:Asymptotic-behavior-and} we obtain results on
the asymptotic behavior of the distribution of $Y_{t}^{t_{0}}$ when
$t_{0}$ is far from the origin as well as when $\alpha\rightarrow1$
and the governing equation of $Y_{t}^{t_{0}}$. 

One example for a process that lies in $\mathcal{S}$ is the Fractional
Poisson Process(FPP) which we denote by $N_{t}^{\alpha}$. The FPP
is a renewal process with interarrival times $W_{n}$ such that $P\left(W_{1}>t\right)=E_{\alpha}\left(-\lambda t^{\alpha}\right)$
where 
\[
E_{\alpha}\left(z\right)=\sum_{k=0}^{\infty}\frac{z^{k}}{\Gamma\left(\alpha k+1\right)}
\]
is the Mittag-Leffler function. Since the interarrival times are not
exponentially distributed the process $N_{t}^{\alpha}=\sup\left\{ k:T_{k}\leq t\right\} $,
where $T_{k}=\sum_{i=1}^{k}W_{i}$ are the arrival times, is not Markovian
and the calculation of the finite dimensional distributions of $N_{t}^{\alpha}$
is no longer straightforward. The FPP was first studied in \cite{Laskin2003},\cite{Jumarie2001}
and \cite{Mainardi2007a,Mainardi2007}. In \cite{Beghin2009} an integral
representation of the one dimensional distribution of the FPP was
given and was used in \cite{Politi2011} to find and simulate the
finite dimensional distributions of the FPP. In \cite{Meerschaert2011},
it was shown that $N_{t}^{\alpha}=N_{E_{t}}$ where $N_{t}$ is a
Poisson process and $E_{t}$ is the inverse of a standard stable subordinator
of index $0<\alpha<1$ independent of $N_{t}$. 

Since the distribution of the increments (and therefore the aging
process) of the CTRWL is closely related to the two dimensional distributions,
their study is quite cumbersome. In a recent paper (\cite{Meerschaert}),
Meerscheart and Straka found a way of embedding CTRWLs in a larger
state space that renders these processes Markovian. We use this method
to find the finite dimensional distributions of the process $Y_{t}^{t_{0}}$,
its asymptotic behavior, self-similarity-like property and its FFPE.

\section{Finite dimensional distribution of CTRWL\label{sec:Finite-dimensional-distribution}}

CTRWL are usually not Markovian, a fact that makes the calculation
of their finite dimensional distributions quite difficult. It is therefore
that the distribution of the increments (which can be obtained by
the finite dimensional distributions) of the CTRWL is not well understood.
\\
Although the method in \cite{Meerschaert} is very general we focus
only on uncoupled CTRWLs which are Levy processes time changed by
the inverse of an independent stable subordinator. In order to facilitate
reading of this section and referring to the original paper we retain
most of the notation in \cite{Meerschaert}. The uncoupled CTRW we
consider consist of two independent sequences of iid r.vs, $\{W_{n}^{c}\}$
and $\{J_{n}^{c}\}$. The parameter $c$ is the convergence parameter
as in \cite{Meerschaert2008} which allows us to construct infinitesimal
triangular arrays. Here, $\{J_{n}^{c}\}$ represents the size of the
jumps of a particle in space, while $\{W_{n}^{c}\}$ represents the
waiting times between jumps. Hence, the time elapsed by the particle's
k'th jump is $T_{k}^{c}=D_{0}^{c}+\sum_{i=1}^{k}W_{i}^{c}$ and the
position of the particle is $S_{k}^{c}=A_{0}^{c}+\sum_{i=1}^{k}J_{i}^{c}$.
Let $L_{t}^{c}=\sup\{k:T_{k}^{c}\leq t\}$ be the number of jumps
until time $t$, then the CTRW $Y_{t}^{c}$ is 
\[
Y_{t}^{c}=A_{0}^{c}+\sum_{i=1}^{L_{t}^{c}}J_{i}^{c}.
\]
Assume we have
\begin{equation}
\left(S_{\left[cu\right]}^{c},T_{\left[cu\right]}^{c}\right)=\left(A_{0}^{c},D_{0}^{c}\right)+\sum_{i=1}^{\left[cu\right]}\left(J_{i}^{c},W_{i}^{c}\right)\Rightarrow\left(A_{u},D_{u}\right)\label{eq:CTRW limit condition}
\end{equation}
where $\Rightarrow$ denotes convergence in the Skorokhod $J_{1}$
topology. In this paper we assume $D_{u}$ is a stable subordinator
of index $0<\alpha<1$ starting from $D_{0}$, i.e, $E\left(e^{-s\left(D_{u}-D_{0}\right)}\right)=e^{-uCs^{\alpha}}$,
where $C$ is a constant. This can be achieved by assuming $W_{i}^{c}=c^{-\frac{1}{\alpha}}W_{i}$
where $\left\{ W_{i}\right\} $ are independent random variables that
are in the strict domain of attraction of $D_{1}-D_{0}$. Note that
$A_{t}-A_{0}$ is a L�vy process as it is the limit of a triangular
array. Now, let $E_{t}=\inf\{s:D_{s}>t\}$ be the first hitting time
of $D_{t}$, also called the inverse of $D_{t}$. By \cite[Theorem 2.4.3]{Straka2011diss}
applied to the case of independent space and time jumps we have
\begin{equation}
Y_{t}^{c}\Rightarrow Y_{t}=A_{E_{t}},\label{eq:CTRW limit theorem}
\end{equation}
as $c\rightarrow\infty$ where convergence is in the Skorokhod $J_{1}$
topology, see also \cite[theorem 3.6]{straka2011} and \cite[Theorem 3.1]{Jurlewicz2012}.
Since $\left(S_{k}^{c},T_{k}^{c}\right)$ is a Markov chain for all
$c>0$ it follows that the CTRWL $Y_{t}$ is a semi-Markov process
and it is possible to embed it in a process of larger state space
that includes the time to regeneration, the \emph{remaining life time
process} $R_{t}$\emph{.} More precisely, let $\mathbb{D}\left([0,\infty),\mathbb{R}^{2}\right)$
be the space of c�dl�g functions $f:[0,\infty)\rightarrow\mathbb{R}^{2}$
with the $J_{1}$ Skorokhod topology which is endowed with transition
operators $T_{u}$, $u>0$ and hence a probability measure $P^{\chi,\tau}$
such that trajectories start at point $\left(\chi,\tau\right)$ with
probability one. Thus, we have a stochastic basis $\left(\Omega,\mathcal{F}_{\infty},\left(\mathcal{F}_{u}\right)_{u\geq0},P^{\chi,\tau}\right)$,
where each element of $\Omega$ is in $\mathbb{D}\left([0,\infty),\mathbb{R}^{2}\right)$,
$\mathcal{F}_{u}=\sigma\left(\left(A_{u}\left(\omega\right),D_{u}\left(\omega\right)\right)\right)$
and $\mathcal{F}_{\infty}=\vee_{u>0}\mathcal{F}_{u}$. The process
$\left(A,D\right)_{t}$ has a generator of the form
\begin{multline}
\mathcal{A}\left(f\right)\left(x,t\right)=b\frac{\partial f(x,t)}{\partial x}-\frac{1}{2}a\frac{\partial^{2}f(x,t)}{\partial x^{2}}\\
+\intop_{\mathbb{R}^{2}}\left(f\left(x+y,t+w\right)-f\left(x,t\right)-y\frac{\partial f(x,t)}{\partial x}1_{\{\left|\left(y,w\right)\right|<1\}}\right)K\left(dy,dw\right),\label{eq:General generator}
\end{multline}
where $a>0$ and $b\in\mathbb{R}$ and $K\left(dy,dw\right)$ is a
L�vy measure. The \emph{occupation time measure} of the process $\left(A,D\right)_{t}$
is the average time spent by the process in a given Borel set in $\mathbb{R}^{2}$,
i.e
\[
\intop f\left(x,t\right)U^{\chi,\tau}\left(dx,dt\right)=\mathbb{E}^{\chi,\tau}\left(\int_{0}^{\infty}f\left(A_{u},D_{u}\right)du\right)=\intop_{0}^{\infty}T_{u}f\left(\chi,\tau\right)du.
\]
{} Let us now define the \emph{remaining life time process} $R_{t}$
\[
R_{t}=D_{E_{t}}-t,
\]
which is the time left for the process $Y_{t}$ to leave its current
state. It was proven in \cite[Theorem 2.3]{Meerschaert} that
\begin{multline}
E^{\chi,\tau}\left(f\left(Y_{t},R_{t}\right)\right)=\\
\integral{x\in\mathbb{R}}{}\integral{\,s\in[\tau,t]}{}U^{\chi,\tau}\left(dx,ds\right)\integral{y\in\mathbb{R}}{}\integral{\,w\in[t-s,\infty)}{}K\left(dy,dw\right)f\left(x+y,w-\left(t-s\right)\right).\label{eq:general expectation of (Y_t,R_t)}
\end{multline}
In \cite{Meerschaert}, a more general CTRWL is considered and hence
a more general form of (\ref{eq:General generator}) where the coefficients
$a$ and $b$ as well as the L�vy measure $K\left(dy,dw\right)$ are
allowed to be dependent on the position of the CTRWL in space and
time, that is, we have $b\left(x,t\right),a\left(x,t\right)$ and
$K\left(x,t;dy,dw\right)$. As was noted in \cite[section 4]{Meerschaert},
when these coefficients do not depend on $t$ (as in our case), the
process $\left(Y_{t},R_{t}\right)$ is a homogeneous Markov process.
More precisely, we define
\begin{align}
Q_{t}\left[f\right]\left(y,0\right) & =E^{y,0}\left(f\left(Y_{t},R_{t}\right)\right)\label{eq:general transition operators1}\\
Q_{t}\left[f\right]\left(y,r\right) & =1_{\{0\leq t<r\}}f\left(y,r-t\right)+1_{\{0\leq r\leq t\}}Q_{t-r}\left[f\right]\left(y,0\right)\qquad r>0,\label{eq:general transition operators2}
\end{align}
for every $f$ bounded and measurable on $\mathbb{R}\times[0,\infty)$.
$Q_{t}$ is the transition operator of the Markov process $\left(Y_{t},R_{t}\right)$
starting at $\chi,\tau$, i.e
\begin{equation}
E^{\chi,\tau}\left(f\left(Y_{t+h},R_{t+h}\right)\mid\sigma\left(\left(Y_{r},R_{r}\right),t\geq r\geq0\right)\right)=Q_{h}\left[f\right]\left(Y_{t},R_{t}\right).\label{eq:Markov property of (Y_t,R_t)}
\end{equation}
\\
One can use the Chapman-Kolmogorov\textquoteright s equation to obtain
the finite dimensional distributions of the process $Y_{t}$. For
example, suppose $\left(Y_{0},R_{0}\right)=\left(0,0\right)$ a.s,
then for the two dimensional distribution of the process $Y_{t}$
at times $t_{1}<t_{2}$ we have
\begin{multline}
P\left(Y_{t_{1}}\in B_{1},Y_{t_{2}}\in B_{2}\right)=P\left(\left(Y_{t_{1}}\in B_{1},R_{t_{1}}\in[0,\infty)\right),\left(Y_{t_{2}}\in B_{2},R_{t_{2}}\in[0,\infty)\right)\right)\\
=Q_{t_{1}}\left[1_{\left\{ B_{1}\times\mathbb{R}\right\} }\left(y_{1},r_{1}\right)Q_{t_{2}-t_{1}}\left[1_{\left\{ B_{2}\times\mathbb{R}\right\} }\left(y_{2},r_{2}\right)\right]\left(y_{1},r_{1}\right)\right]\left(0,0\right),\label{eq:Two dimensiional distribution}
\end{multline}
where $B_{1},B_{2}\in\mathbb{\mathcal{B}\left(R\right)}$ are Borel
sets.
\begin{rem}
In \cite{Meerschaert} a result stronger than (\ref{eq:Markov property of (Y_t,R_t)})
was shown. Indeed, the process $\left(Y_{t},R_{t}\right)$ is a strong
Markov process with respect to a filtration larger than the natural
filtration. For the sake of brevity and the fact that the Markov property
is adequate for our work we brought the result in a weaker form.
\end{rem}

\section{Aging\label{sec:aging} }

Let us assume (\ref{eq:CTRW limit condition}) holds with $\chi=\tau=0$
so $A_{t}$ is a Levy process with CDF $P_{t}\left(x\right)=P\left(A_{t}\in(-\infty,x]\right)$
and with Levy triplet $\left(\mu,A,\phi\right)$, i.e
\[
E\left(e^{iuA_{t}}\right)=\exp\left[t\left(i\mu u-\frac{1}{2}Au^{2}+\integral{\mathbb{R}}{}\left(e^{iuy}-1-iuy1_{\{\left|y\right|<1\}}\right)\phi\left(dy\right)\right)\right].
\]
 Also assume $D_{t}$ is a stable subordinator of index $0<\alpha<1$
with Laplace transform (LT) $E\left(e^{-uD_{t}}\right)=e^{-tcu^{\alpha}}$
independent of $A_{t}$. Then (\ref{eq:General generator}) holds
with $b=\mu$, $a=A$ and (see \cite[Corollary 2.3]{becker2004limit})
\\
\begin{equation}
K\left(dy,dw\right)=\phi\left(dy\right)\delta_{0}\left(dw\right)+\delta_{0}\left(dy\right)\frac{c\alpha}{\Gamma\left(1-\alpha\right)}w^{-1-\alpha}1_{\{w>0\}}dw.\label{eq:S class Levy measure}
\end{equation}
 %
Next, we wish to find the occupation measure of the process $\left(A,D\right)_{t}$.
We have for $f\left(y,w\right)=1_{\{(-\infty,x]\times(-\infty,t]\}}\left(y,w\right)$
\begin{alignat*}{1}
\intop f\left(y,w\right)U^{\chi,\tau}\left(dy,dw\right) & =\mathbb{E}^{\chi,\tau}\left(\integral 0{\infty}f\left(A_{u},D_{u}\right)du\right)\\
=\integral 0{\infty}T_{u}f\left(\chi,\tau\right)du & =\integral 0{\infty}T_{u}1_{\{(-\infty,x]\times(-\infty,t]\}}\left(\chi,\tau\right)du\\
 & =\integral 0{\infty}\integral{w\in\mathbb{R}}{}\integral{\,y\in\mathbb{R}}{}1_{\{(-\infty,x]\times(-\infty,t]\}}\left(y+\chi,w+\tau\right)q_{u}(dy,dw)du,
\end{alignat*}
where $q_{t}$ is the distribution of the process $\left(A,D\right)_{t}$
cf. \cite[Eq. 3.11]{Applebaum2009}. By independence of $A_{t}$ and
$D_{t}$ we have
\begin{alignat}{1}
\intop f\left(x,t\right)U^{\chi,\tau}\left(dx,dt\right) & =\intop_{0}^{\infty}P\left(A_{u}\in(-\infty,x-\chi]\right)P\left(D_{u}\in(-\infty,t-\tau]\right)du\nonumber \\
 & =\intop_{0}^{\infty}P_{u}\left(x-\chi\right)\intop_{-\infty}^{t-\tau}g\left(w,u\right)dwdu,\label{eq:implicit occupation measure}
\end{alignat}
where $g\left(x,t\right)$ is the pdf of $D_{t},$ i.e $g\left(x,t\right)dx=P\left(D_{t}\in dx\right)$
and is known to be absolutely continuous with respect to the Lebesgue
measure \cite[Section 2.4]{Zolotarev86}. 

Since $\left(A,D\right)_{t}$ is a Levy process the coefficients in
(\ref{eq:General generator}) are independent of $t$ and therefore
the process $\left(A,D\right)_{t}$ is a Markov additive process \cite[Section 4]{Meerschaert}
and the occupation measure is of the form 
\begin{equation}
U^{y}\left(dx,dt\right)=\intop_{0}^{\infty}P_{u}\left(dx-y\right)g\left(t,u\right)dudt.\label{eq:S class occupation measure}
\end{equation}
Furthermore, one may choose $\tau=0$ and plug (\ref{eq:S class Levy measure})
and (\ref{eq:S class occupation measure}) in (\ref{eq:general expectation of (Y_t,R_t)})
to obtain
\begin{alignat}{1}
E^{\chi,0}\left(f\left(Y_{t},R_{t}\right)\right) & =\integral{x\in\mathbb{R}}{}\integral{\,s\in[0,t]}{}\left(\integral{u\in\mathbb{R}^{+}}{}P_{u}\left(dx-\chi\right)g\left(s,u\right)du\right)\label{eq:Expectation for S class}\\
 & \times\integral{y\in\mathbb{R}}{}\integral{\,w\in[t-s,\infty)}{}\left(\phi\left(dy\right)\delta_{0}\left(dw\right)+\delta_{0}\left(dy\right)\frac{c\alpha}{\Gamma\left(1-\alpha\right)}w{}^{-1-\alpha}dw\right)f\left(x+y,w-\left(t-s\right)\right)ds\nonumber \\
 & =\integral{x\in\mathbb{R}}{}\integral{\,s\in[0,t]}{}\left(\integral{u\in\mathbb{R}^{+}}{}P_{u}\left(dx-\chi\right)g\left(s,u\right)du\right)\nonumber \\
 & \times\integral{\,w\in[t-s,\infty)}{}f\left(x,w-\left(t-s\right)\right)\frac{c\alpha}{\Gamma\left(1-\alpha\right)}w{}^{-1-\alpha}dwds,\nonumber 
\end{alignat}
for $Y_{t}\in\mathcal{S}$ and its time to regeneration $R_{t}$.
\\
 We say that the r.v $X$ has beta distribution with parameters $\mu,\nu>0$
if it has pdf of the form
\[
f\left(x,\mu,\nu\right)=\frac{x^{\mu-1}\left(1-x\right)^{\nu-1}}{B\left[\mu,\nu\right]}\qquad x\in\left(0,1\right)
\]
where $B\left[\mu,\nu\right]=\frac{\Gamma\left(\mu\right)\Gamma\left(\nu\right)}{\Gamma\left(\mu+\nu\right)}$
is the Beta function and we write $X\sim B\left(\mu,\nu\right)$ .
We say that the r.v $X$ has beta prime distribution with parameters
$\mu,\nu>0$ if it has pdf of the form
\begin{equation}
f\left(x,\mu,\nu\right)=\frac{x^{\mu-1}\left(1+x\right)^{-\mu-\nu}}{B\left[\mu,\nu\right]}\qquad x>0\label{eq:Beta prime distribution}
\end{equation}
and we write $X\sim B'\left(\mu,\nu\right)$. It was noted in \cite[II.4]{feller1971introduction}
that if $X\sim B\left(\mu,\nu\right)$ then $\frac{X}{1-X}\sim B'\left(\mu,\nu\right)$.
The distribution (\ref{eq:Beta prime distribution}) can be further
generalized to the so called \emph{generalized Beta prime distribution}
also known as the general Beta of the second kind distribution whose
pdf is 
\begin{equation}
f\left(x,\mu,\nu,h\right)=\frac{\left(\frac{x}{h}\right)^{\mu-1}\left(1+\frac{x}{h}\right)^{-\mu-\nu}}{h\cdot B\left[\mu,\nu\right]}\qquad x>0\label{eq:generalized beta prime distribution}
\end{equation}
with $h,\mu,\nu>0$ . If $X$ has generalized Beta prime distribution
of the form (\ref{eq:generalized beta prime distribution}) then we
write $X\sim GB2\left(\mu,\nu,h\right)$.

\begin{thm}
\label{thm: aging property}Let $Y_{t}^{t_{0}}=A_{E_{t+t_{0}}}-A_{E_{t_{0}}}$
where $t_{0}>0$ be the aging process. Let $B_{1},B_{2},...,B_{k}$
be Borel sets such that $0\notin B_{1}$. Let $p_{t_{0}}\left(r\right)=f\left(r,1-\alpha,\alpha,t_{0}\right)$
be a generalized beta prime distribution as in (\ref{eq:generalized beta prime distribution}).
Then we have for $0<t_{1}<t_{2}<\cdots<t_{k}$
\begin{equation}
P\left(Y_{t_{1}}^{t_{0}}\in B_{1},Y_{t_{2}}^{t_{0}}\in B_{2},...,Y_{t_{k}}^{t_{0}}\in B_{k}\right)=\integral 0{t_{1}}P\left(Y_{t_{1}-r}\in B_{1},Y_{t_{2}-r}\in B_{2},...,Y_{t_{k}-r}\in B_{k}\right)p_{t_{0}}\left(r\right)dr.\label{eq:Aging result}
\end{equation}
 \end{thm}
\begin{proof}
For simplicity, we proof the result for $k=2$ , the proof for $k>2$
is similar. We have
\begin{alignat}{1}
P\left(Y_{t_{1}}^{t_{0}}\in B_{1},Y_{t_{2}}^{t_{0}}\in B_{2}\right) & =Q_{t_{0}}\left[1_{\left\{ \mathbb{R}\times\mathbb{R}\right\} }\left(y_{0},r_{0}\right)\right.\nonumber \\
 & \times\left.Q_{t_{1}}\left[1_{\left\{ B_{1}+y_{0}\times\mathbb{R}\right\} }\left(y_{1},r_{1}\right)Q_{t_{2}-t_{1}}\left[1_{\left\{ B_{2}+y_{0}\times\mathbb{R}\right\} }\left(y_{2},r_{2}\right)\right]\left(y_{1},r_{1}\right)\right]\left(y_{0},r_{0}\right)\right]\left(0,0\right).\label{eq:two dimensional distribution}
\end{alignat}
It is easy to see that by (\ref{eq:Expectation for S class}) the
semi-group operator $Q_{t}$ is translation invariant with respect
to the space variable when $r=0$, i.e, $Q_{t}\left[f\right]\left(y+a,0\right)=Q\left[g\right]\left(y,0\right)$
where $g\left(y,r\right)=f\left(y+a,r\right)$. Moreover, 
\begin{alignat*}{1}
Q_{t}\left[f\right]\left(y+a,r\right) & =1_{\left\{ 0\leq t<r\right\} }f\left(y+a,r-t\right)+1_{\left\{ 0\leq r\leq t\right\} }Q_{t-r}\left[f\right]\left(y+a,0\right)\\
 & =1_{\left\{ 0\leq t<r\right\} }g\left(y,r-t\right)+1_{\left\{ 0\leq r\leq t\right\} }Q_{t-r}\left[g\right]\left(y,0\right)\\
 & =Q_{t}\left[g\right]\left(y,r\right).
\end{alignat*}
Hence, $Q_{t}$ is translation invariant with respect to the space
variable. Consequently, since $0\notin B_{1}$, by (\ref{eq:general transition operators2})
we have

\begin{alignat}{1}
 & Q_{t_{1}}\left[1_{\left\{ B_{1}+y_{0}\times\mathbb{R}\right\} }\left(y_{1},r_{1}\right)Q_{t_{2}-t_{1}}\left[1_{\left\{ B_{2}+y_{0}\times\mathbb{R}\right\} }\left(y_{2},r_{2}\right)\right]\left(y_{1},r_{1}\right)\right]\left(y_{0},r_{0}\right)\nonumber \\
 & =1_{\left\{ 0\leq r_{0}\leq t_{1}\right\} }Q_{t_{1}-r_{0}}\left[1_{\left\{ B_{1}+y_{0}\times\mathbb{R}\right\} }\left(y_{1}+y_{0},r_{1}\right)Q_{t_{2}-t_{1}}\left[1_{\left\{ B_{2}+y_{0}\times\mathbb{R}\right\} }\left(y_{2},r_{2}\right)\right]\left(y_{1}+y_{0},r_{1}\right)\right]\left(0,0\right)\nonumber \\
 & =1_{\left\{ 0\leq r_{0}\leq t_{1}\right\} }Q_{t_{1}-r_{0}}\left[1_{\left\{ B_{1}\times\mathbb{R}\right\} }\left(y_{1},r_{1}\right)Q_{t_{2}-t_{1}}\left[1_{\left\{ B_{2}+y_{0}\times\mathbb{R}\right\} }\left(y_{2}+y_{0},r_{2}\right)\right]\left(y_{1},r_{1}\right)\right]\left(0,0\right)\nonumber \\
 & =1_{\left\{ 0\leq r_{0}\leq t_{1}\right\} }P\left(Y_{t_{1}-r_{0}}\in B_{1},Y_{t_{2}-r_{0}}\in B_{2}\right).\label{eq:two dimensional distribution2}
\end{alignat}
For ease of notation we write $P\left(Y_{t_{1}-r_{0}}\in B_{1},Y_{t_{2}-r_{0}}\in B_{2}\right)=f\left(r_{0}\right)$.
Plug (\ref{eq:two dimensional distribution2}) in (\ref{eq:two dimensional distribution})
and use (\ref{eq:Expectation for S class}) to obtain,
\begin{alignat}{1}
P\left(Y_{t_{1}}^{t_{0}}\in B_{1},Y_{t_{2}}^{t_{0}}\in B_{2}\right) & =\integral{s'\in[0,t_{0}]}{}\left(\integral{u'\in\mathbb{R}^{+}}{}g\left(s',u'\right)du'\right)\label{eq:general increnents2}\\
 & \times\integral{w'\in[t_{0}-s',\infty)}{}\frac{c\alpha}{\Gamma\left(1-\alpha\right)}w{}^{-1-\alpha}dwds'\nonumber \\
 & \times\left[1_{\{0\leq w'-\left(t_{0}-s'\right)\leq t_{1}\}}\times f\left(w'-\left(t_{0}-s'\right)\right)\right]\nonumber 
\end{alignat}
\begin{alignat*}{1}
 & =\integral{s'\in[0,t_{0}]}{}\left(\integral{u'\in\mathbb{R}^{+}}{}g\left(s',u'\right)du'\right)\times\integral{w'\in[t_{0}-s',t_{1}+t_{0}-s')}{}f\left(w'-\left(t_{0}-s'\right)\right)\\
 & \times\frac{c\alpha}{\Gamma\left(1-\alpha\right)}w'^{-1-\alpha}dw'ds'.
\end{alignat*}
By \cite[Eq. 37.12]{sato1999levy} if $D_{t}$ is a stable subordinator
of index $0<\alpha<1$ with $E\left(e^{-uX_{t}}\right)=e^{-tcu^{\alpha}}$
and probability distribution $P\left(D_{t}\in dx\right)=g\left(x,t\right)dx$
then its potential density is given by 
\begin{equation}
v\left(s\right)=\integral{u\in\mathbb{R}^{+}}{}g\left(s,u\right)du=\frac{1}{c\Gamma\left(\alpha\right)}s^{\alpha-1}\qquad s>0.\label{eq:potential density stable subordinator}
\end{equation}
 Substitute (\ref{eq:potential density stable subordinator}) in (\ref{eq:general increnents2})
and apply the change of variables $r=w'+s'-t_{0}$ to obtain
\begin{alignat}{1}
P\left(Y_{t_{1}}^{t_{0}}\in B_{1},Y_{t_{2}}^{t_{0}}\in B_{2}\right) & =\integral 0{t_{1}}\integral{s'\in[0,t_{0}]}{}\frac{s'^{\alpha-1}}{c\Gamma\left(\alpha\right)}f\left(r\right)\frac{c\alpha}{\Gamma\left(1-\alpha\right)}\left(r-s'+t_{0}\right)^{-1-\alpha}ds'dr.\label{eq:general increments}
\end{alignat}
Now apply the change of variables $v=s'\left(r-s'+t_{0}\right)^{-1}$
to compute the integral with respect to $s'$ and to obtain 
\begin{alignat*}{1}
P\left(Y_{t_{1}}^{t_{0}}\in B_{1},Y_{t_{2}}^{t_{0}}\in B_{2}\right) & =\integral 0{t_{1}}f\left(r\right)\frac{\left(\frac{r}{t_{0}}\right)^{-\alpha}\left(1+\frac{r}{t_{0}}\right)^{-1}}{t_{0}\cdot B\left[\alpha,1-\alpha\right]}dr\\
 & =\integral 0{t_{1}}P\left(Y_{t_{1}-r}\in B_{1},Y_{t_{2}-r}\in B_{2}\right)p_{t_{1}}\left(r\right)dr.
\end{alignat*}
\end{proof}
\begin{rem}
\label{remark:zero increment probability}It follows from Theorem
\ref{thm: aging property} that 
\begin{equation}
P\left(Y_{t}^{t_{0}}=0\right)=\intop_{t}^{\infty}p_{t_{0}}\left(r\right)dr+\integral 0tP\left(Y_{t-r}=0\right)p_{t_{0}}\left(r\right)dr>0.\label{eq:one dimensional zero probability}
\end{equation}
Therefore, the distribution of $Y_{t}^{t_{0}}$ has an atom at the
origin for every $t$. More interesting is the fact that if $P\left(A_{t}=0\right)=0$(this
is true for all processes with pdf) then $P\left(Y_{t}^{t_{0}}=0\right)$
does not depend on the choice of the process $A_{t}$. On the other
hand it can be easily seen that for every $t$ the process $Y_{t}^{t_{0}}$
has density on $\mathbb{R}\left\backslash \{0\}\right.$ given by
$p_{t_{0}}\left(x,t\right)=\integral 0tp\left(x,t-r\right)p_{t_{0}}\left(r\right)dr$
whenever $A_{t}$ has pdf $p\left(x,t\right)$. Furthermore, note
that the finite dimensional distributions of the process $Y^{t_{0}}$
on Borel sets $B_{1},...,B_{k}$ such that $0\notin B_{1}$, determine
completely the finite dimensional distributions of the process $Y^{t_{0}}$.
We demonstrate this for $k=2$; if $B_{2}$ is a Borel set then 
\[
P\left(Y_{t_{1}}^{t_{0}}=0,Y_{t_{2}}^{t_{0}}\in B_{2}\right)=P\left(Y_{t_{2}}^{t_{0}}\in B_{2}\right)-P\left(Y_{t_{1}}^{t_{0}}\in\mathbb{R}/\{0\},Y_{t_{2}}^{t_{0}}\in B_{2}\right),
\]
which by (\ref{eq:one dimensional zero probability}) determines the
two dimensional distributions completely.
\end{rem}

\begin{rem}
In \cite{Barkai2003}, a result similar to Theorem \ref{thm: aging property}
for the one dimensional distribution is obtained for CTRW for large
$t_{0}$ and $t$. The proof in \cite{Barkai2003} sheds light on
our result, as it was derived from showing that the distribution of
the first epoch $\tau_{1}$ of the aging CTRW $X_{t}^{t_{0}}$ has
beta prime distribution, i.e $\tau_{1}\sim B'\left(1-\alpha,\alpha,t_{0}\right)$.
This can be shown by a result by Dynkin on renewal processes (\cite[Theorem 8.6.3]{bingham1989regular}).
Interestingly, the distribution of the first epoch $\tau_{1}$ does
not scale out as we move to the limit and obtain the process $Y_{t}^{t_{0}}$.
Indeed, one can show (similarly to the proof of Theorem \ref{thm: aging property})
that the distribution of the process $R_{t}$, the time left before
the next regeneration at time $t$, is
\begin{equation}
f_{R_{t}}\left(r\right)=\frac{\left(\frac{r}{t}\right)^{-\alpha}\left(1+\frac{r}{t}\right)^{-1}}{t\cdot B\left[\alpha,1-\alpha\right]}\qquad r>0.\label{eq:the pdf of R_t}
\end{equation}
 Since it was noted in \cite{Meerschaert} that the process $Y_{t}$
starts afresh at time $H_{t}=D_{E_{t}}=t$ depending only on the position
of $Y_{t}$, and by the fact that in our case the process $Y_{t}$
is homogeneous in space, it follows that once the process $Y_{t}^{t_{0}}$
leaves the state $0$ it behaves like the process $Y_{t}$ from that
point on. Now, condition the probability $P\left(Y_{t_{1}}^{t_{0}}\in B_{1},Y_{t_{2}}^{t_{0}}\in B_{2},...,Y_{t_{k}}^{t_{0}}\in B_{k}\right)$
on the event $\left\{ R_{t_{0}}=r\right\} $and integrate with respect
to $r$ to obtain (\ref{eq:Aging result}). It should be clear now
why $0\notin B_{1}$ as we would like to make sure that the system
is mobilized before time $t_{1}$.
\end{rem}

\begin{rem}
Let $X_{t}$ be a renewal process with interarrival times $\{W_{i}\}$
whose tail distribution $1-F\left(x\right)\in R\left(-\alpha\right)$
for $0<\alpha<1$, namely, there exists a slowly varying function
$L\left(x\right)$ such that $1-F\left(x\right)\sim x^{-\alpha}L\left(x\right)$
when $x\rightarrow\infty$. Define the arrival times $T_{n}=\sum_{i=1}^{n}W_{i}$
and let $S_{t}=t-T_{X_{t}}$ be the \emph{age process, }the time spent
at the current state. It was shown in \cite[Theorem 8.6.3]{bingham1989regular}
that the distribution of $\frac{S_{t}}{t}$ converges, as $t\rightarrow\infty$.
The limit is the so called Generalized Beta of the first kind distribution
$GB1\left(1-\alpha,\alpha,1\right)$ whose pdf equals $f_{V_{1}}$,
where 
\begin{equation}
f_{V_{t}}\left(v\right)=\frac{\left(\frac{v}{t}\right)^{-\alpha}\left(1-\frac{v}{t}\right)^{\alpha-1}}{tB\left[\alpha,1-\alpha\right]}\qquad0<v<t.\label{eq: the pdf of V_t}
\end{equation}
 In \cite{Meerschaert}, the analogous process $V_{t}=t-D_{E_{t}-}$
was defined to track the time that has passed since the last regeneration
of the process $Y_{t}$. It can be easily shown, along similar lines
to the proof of Theorem \ref{thm: aging property}, that the process
$V_{t-}$ has distribution $GB1\left(1-\alpha,\alpha,t\right)$. Equations
\ref{eq:the pdf of R_t} and \ref{eq: the pdf of V_t} explain the
results of Jurlewicz et al in \cite{Jurlewicz2012}. There it was
proven (\cite[Eq. 5.12]{Jurlewicz2012}) that $D_{E_{t}}$ has pdf
\begin{equation}
g\left(r\right)=\frac{r^{-1}}{B\left[\alpha,1-\alpha\right]}\left(\frac{t}{r-t}\right)^{\alpha}\qquad r>t,\label{eq: Jurlewicz eq 5.9}
\end{equation}

and that $D_{E_{t}-}$ has pdf (\cite[Eq. 5.9]{Jurlewicz2012}) 
\begin{equation}
h\left(v\right)=\frac{v^{\alpha-1}\left(t-v\right)^{-\alpha}}{B\left[\alpha,1-\alpha\right]}\qquad0<v<t.\label{eq:Jurlewicz eq 5.12}
\end{equation}
Equation \ref{eq: Jurlewicz eq 5.9} and \ref{eq:Jurlewicz eq 5.12}
can be obtained by \ref{eq:the pdf of R_t} and \ref{eq: the pdf of V_t}
respectively, by translation and reflection.
\end{rem}

\section{Aging self similarity\label{sec:Aging-self-similarity}}

Recall that a process $X_{t}$ is called self-similar if for every
$a>0$ there exists $b>0$ such that the finite dimensional distributions
of the time scaled process $X_{at}$ equals that of the process $bX_{t}$.
It is well known (\cite[Section 13]{sato1999levy}) that if $X_{t}$
is a L�vy process then it is self-similar if and only if $X_{t}$
is strictly stable, i.e for every $a>0$ there exist $b>0$ such that
$E\left(e^{iuX_{1}}\right)^{a}=E\left(e^{iubX_{1}}\right)$. For self-similar
non trivial processes that are stochastically continuous at $t=0$,
$b=a^{H}$ (\cite[Theorem 1.1.1]{embrechts2009selfsimilar}), where
$H>0$ if and only if $X_{t}=0$ with probability one. $H$ is sometimes
called the Hurst parameter. For example, for fractional Brownian motion
$0<H<1$ while the Hurst parameter of the stable subordinator of index
$0<\alpha\leq2$ is $\nicefrac{1}{\alpha}$ . For self-similar processes
with stationary increments and finite second moment the Hurst parameter
(when it exists) determines long range dependence (\cite[Section 3.2]{embrechts2009selfsimilar}).
Throughout this section we consider the process $Y_{t}^{t_{0}}=A_{E_{t+t_{0}}}-A_{E_{t_{0}}}$
where $A_{t}$ is a strictly stable process whose Hurst parameter
we denote by $\nicefrac{1}{\beta}$ and $E_{t}$ is the inverse of
a stable subordinator of index $\alpha.$ We wish to find whether
$Y_{t}^{t_{0}}$ has the property of self-similarity or a different
property that resembles self-similarity to some extent. From Theorem
\ref{thm: aging property} it is only reasonable that any self-similarity-like
property of $Y_{t}^{t_{0}}$ should be strongly connected to the self-similarity
of the process $Y_{t}$.

The next corollary states that although the aging process $Y_{t}^{t_{0}}$
is not self-similar it exhibits a self-similar-like behavior. Intuitively
it suggests that $Y_{at}^{t_{0}}$ behaves like a ``younger''($a>1)$
scaled version of itself.
\begin{cor}
Let $Y_{t}^{t_{0}}$be an aging process and let $B_{i}$ for $1\leq i\leq k$
be Borel sets in $\mathbb{R}$. Then
\[
\left(Y_{at_{1}}^{t_{0}},Y_{at_{2}}^{t_{0}},...,Y_{at_{k}}^{t_{0}}\right)\overset{d}{=}\left(a^{\frac{\alpha}{\beta}}Y_{t_{1}}^{\frac{t_{0}}{a}},a^{\frac{\alpha}{\beta}}Y_{t_{2}}^{\frac{t_{0}}{a}},...,a^{\frac{\alpha}{\beta}}Y_{t_{k}}^{\frac{t_{0}}{a}}\right).
\]
\end{cor}
\begin{proof}
For simplicity we only prove the result for $k=2$ as the proof for
$k>2$ is similar. First assume that $B_{1}\subseteq\mathbb{R}$ does
not contain zero . By Theorem \ref{thm: aging property} we have
\[
P\left(Y_{at_{1}}^{t_{0}}\in B_{1},Y_{at_{2}}^{t_{0}}\in B_{2}\right)=\integral 0{at_{1}}P\left(Y_{at_{1}-r}\in B_{1},Y_{at_{2}-r}\in B_{2}\right)p_{t_{0}}\left(r\right)dr.
\]
Apply the change of variables $r'=\frac{r}{a}$ to obtain
\begin{alignat*}{1}
P\left(Y_{at_{1}}^{t_{0}}\in B_{1},Y_{at_{2}}^{t_{0}}\in B_{2}\right) & =\integral 0{t_{1}}P\left(Y_{a\left(t_{1}-r'\right)}\in B_{1},Y_{a\left(t_{2}-r'\right)}\in B_{2}\right)\frac{\left(\frac{r'a}{t_{0}}\right)^{-\alpha}\left(1+\frac{r'a}{t_{0}}\right)^{-1}}{t_{0}\cdot B\left[\alpha,1-\alpha\right]}adr'.
\end{alignat*}
By \cite[Corollary 4.1]{Meerschaert2004} $Y_{t}$ is self similar
with Hurst parameter $\frac{\alpha}{\beta}$. Therefore we have 
\begin{alignat}{1}
P\left(Y_{at_{1}}^{t_{0}}\in B_{1},Y_{at_{2}}^{t_{0}}\in B_{2}\right) & =\integral 0{t_{1}}P\left(a^{\frac{\alpha}{\beta}}Y_{\left(t_{1}-r'\right)}\in B_{1},a^{\frac{\alpha}{\beta}}Y_{\left(t_{2}-r'\right)}\in B_{2}\right)p_{\frac{t_{0}}{a}}\left(r\right)dr\label{eq:aging self similarity two dimensional}\\
 & =P\left(a^{\frac{\alpha}{\beta}}Y_{t_{1}}^{\frac{t_{0}}{a}}\in B_{1},a^{\frac{\alpha}{\beta}}Y_{t_{2}}^{\frac{t_{0}}{a}}\in B_{2}\right).\nonumber 
\end{alignat}
Now, by Remark \ref{remark:zero increment probability} it follows
that (\ref{eq:aging self similarity two dimensional}) holds for any
Borel sets $B_{1},B_{2}\subseteq\mathbb{R}$ and the result follows. 
\end{proof}

\section{Asymptotic behavior and the Fractional Fokker-Planck equation\label{sec:Asymptotic-behavior-and}}

An easy yet important consequence of Theorem \ref{thm: aging property}
is the following.
\begin{cor}
\textup{\label{Corollary asymptotic}Let $B\subseteq\mathbb{R}$ be
a Borel measurable subset such that $0\notin B$ and $P\left(Y_{t}^{t_{0}}\in B\right)\neq0$,
then 
\begin{equation}
P\left(Y_{t}^{t_{0}}\in B\right)\sim Ct_{0}^{\alpha-1}\qquad t_{0}\rightarrow\infty\label{eq:asymptotics of increments}
\end{equation}
where $C=\frac{\sin\left(\pi\alpha\right)}{\pi}\integral 0tP\left(Y_{t-r}\in B\right)r^{-\alpha}dr$.}\end{cor}
\begin{proof}
{} First note that by the continuity of $P\left(Y_{t}\in B\right)$
(see (\ref{lem: continuity of probability})) $C\neq0\Leftrightarrow P\left(Y_{t}^{t_{0}}\in B\right)\neq0$.
By dominated convergence we then have,
\begin{alignat*}{1}
\underset{t_{0}\rightarrow\infty}{\lim}\frac{P\left(Y_{t}^{t_{0}}\in B\right)}{Ct_{0}^{\alpha-1}}=\underset{t_{0}\rightarrow\infty}{\lim}\frac{t_{0}^{\alpha-1}\frac{\sin\left(\pi\alpha\right)}{\pi}\integral 0tP\left(Y_{t-r}\in B\right)r^{-\alpha}\left(1+\frac{r}{t_{0}}\right)^{-1}dr}{t_{0}^{\alpha-1}\frac{\sin\left(\pi\alpha\right)}{\pi}\integral 0tP\left(Y_{t-r}\in B\right)r^{-\alpha}dr} & =1.
\end{alignat*}
\end{proof}
\begin{rem}
When the process $Y_{t}$ is a renewal process (that is the case for
the FPP) the convergence of $P\left(Y_{t}^{t_{0}}\in B\right)$ to
zero is expected by the Renewal Theorem (\cite[XI.1]{feller1971introduction})
and the fact that the interarrival times have the Mittag-Leffer distribution
with infinite expectation. Interestingly, it was shown by Erickson
in \cite[Theorem 1]{Erickson1970}, that if $Y_{t}$ is a renewal
process with interarrival times $W_{n}$ with $F\left(t\right)=P\left(W_{1}\leq t\right)$
such that $1-F\left(t\right)\in R\left(-\alpha\right)$ for $0<\alpha<1$,
i.e $1-F\left(t\right)\sim t^{-\alpha}L\left(t\right)$ as $t\rightarrow\infty$
where $L\left(t\right)$ is a slowly varying function and $F$ is
not arithmetic, then
\begin{equation}
E\left(Y_{t}^{t_{0}}\right)\sim\frac{\sin\left(\pi\alpha\right)}{\pi}\frac{t}{L\left(t_{0}\right)}t_{0}^{\alpha-1}\qquad t_{0}\rightarrow\infty.\label{eq:Erickson}
\end{equation}
We now show how (\ref{eq:Erickson}) can be obtained for the FPP by
Corollary \ref{Corollary asymptotic}. First note that by similar
arguments as in Corollary \ref{Corollary asymptotic} we have
\begin{equation}
E\left(Y_{t}^{t_{0}}\right)\sim t_{0}^{\alpha-1}\frac{\sin\left(\pi\alpha\right)}{\pi}\integral 0tE\left(Y_{t-r}\right)r^{-\alpha}dr\qquad t_{0}\rightarrow\infty.\label{eq:Expectation asymptotics}
\end{equation}
 Let $Y_{t}=N_{t}^{\alpha}$ be the fractional Poisson process with
intensity $\lambda=1.$ By \cite[Eq. 2.7]{Beghin2009}, $E\left(Y_{t-r}\right)=\frac{\left(t-r\right)^{\alpha}}{\Gamma\left(1+\alpha\right)}$
and so by (\ref{eq:Expectation asymptotics}) we have 
\begin{alignat*}{1}
E\left(Y_{t}^{t_{0}}\right) & \sim t_{0}^{\alpha-1}\frac{\sin\left(\pi\alpha\right)}{\pi}\integral 0t\frac{\left(t-r\right)^{\alpha}}{\Gamma\left(1+\alpha\right)}r^{-\alpha}dr\\
 & =t_{0}^{\alpha-1}\frac{\sin\left(\pi\alpha\right)}{\pi}t\Gamma\left(1-\alpha\right).
\end{alignat*}
To see this, note that $\integral 0t\left(t-r\right)^{\alpha}r^{-\alpha}dr=\frac{\Gamma\left(1-\alpha\right)}{\left(\alpha+1\right)}\partial_{t}^{\alpha}\left[t^{\alpha+1}1_{t\geq0}\right]=\Gamma\left(1-\alpha\right)\Gamma\left(\alpha+1\right)t$,
where $\partial_{t}^{\alpha}$ is the Caputo derivative of index $\alpha$
(\ref{eq:Caputo derivative}). This agrees with (\ref{eq:Erickson}).
Indeed, note that by \cite{Repin2000} the asymptotic behavior of
the Mittag-Leffler distribution pdf is $f^{\alpha}\left(t\right)\sim\frac{t^{-1-\alpha}\alpha}{\Gamma\left(1-\alpha\right)}$
as $t\rightarrow\infty$ (note that there is a typo there as $\alpha$
should be in the numerator) and by the Karamata Tauberian Theorem
(\cite[Theorem 1.5.11]{bingham1989regular}) we see that $E^{\alpha}\left(-t^{\alpha}\right)=\integral t{\infty}f\left(y\right)dy\sim\frac{t^{-\alpha}}{\Gamma\left(1-\alpha\right)}$
as $t\rightarrow\infty$ so $\left(L\left(t\right)\right)^{-1}=\Gamma\left(1-\alpha\right)$.
\end{rem}
While it is known that generally CTRWL lose their stationarity property
for $0<\alpha<1$ (\cite[Corollary 4.3]{Meerschaert2004}), Theorem
\ref{thm: aging property} suggests a way of measuring the stationarity
of a process in the class $\mathcal{S}$. The FPP for example has
no stationary increments for $0<\alpha<1$, however, for $\alpha=1$
we obtain the Poisson process which is of course stationary as being
a Levy process. We proceed with a useful lemma that states that the
distribution of the processes in $\mathcal{S}$ is continuous as a
function of time. 
\begin{lem}
\label{lem: continuity of probability}Let $Y_{t}\in\mathcal{S}$
and $C\subset\mbox{\ensuremath{\mathbb{R}}}$ a Borel set, then the
function $t\mapsto P\left(Y_{t}\in C\right)$ is continuous on $\left(0,\infty\right)$.\end{lem}
\begin{proof}
Since $Y_{t}=A_{E_{t}}$, by a simple conditioning argument (\cite[Eq. (2.7)]{Meerschaert2013})
we have
\[
P\left(Y_{t}^ {}\in C\right)=\integral 0{\infty}P\left(A_{y}\in C\right)h\left(y,t\right)dy,
\]
where $h\left(x,t\right)$ is the pdf of the process $E_{t}.$ Then
\begin{alignat*}{1}
\underset{h\rightarrow0}{\lim\sup}\left|P\left(Y_{t+h}\in C\right)-P\left(Y_{t}\in C\right)\right| & =\underset{h\rightarrow0}{\lim\sup}\left|\integral 0{\infty}P\left(A_{y}\in C\right)h\left(y,t+h\right)dy-\integral 0{\infty}P\left(A_{y}\in C\right)h\left(y,t\right)dy\right|\\
 & \leq\underset{h\rightarrow0}{\lim\sup}\integral 0{\infty}\left|h\left(y,t+h\right)-h\left(y,t\right)\right|dy.
\end{alignat*}
It was proved in (\cite[Corollary 3.1]{Meerschaert2004}) that 
\[
h\left(x,t\right)=\frac{t}{\alpha}x^{-1-\frac{1}{\alpha}}g\left(tx^{-\frac{1}{\alpha}}\right)
\]
where $g\left(x\right)$ is the pdf of a stable r.v. Since $g\left(x\right)$
is smooth it follows that $h\left(x,t\right)$ is continuous on $t,x>0$
. Trivially we have 
\[
\underset{h\rightarrow0}{\lim}\integral 0{\infty}h\left(y,t+h\right)dy=\integral 0{\infty}h\left(y,t\right)dy=1.
\]
Hence, a basic result in analysis \cite[Chapter 7, Theorem 7]{royden2010real}
implies that
\[
\underset{h\rightarrow0}{\lim}\integral 0{\infty}\left|h\left(y,t+h\right)-h\left(y,t\right)\right|dy=0,
\]
and the result follows. 
\end{proof}
The next result states that as $\alpha\rightarrow1$ the process $Y_{t}$
, in some sense, becomes more stationary. 
\begin{prop}
Let $Y_{t}\in\mathcal{S}$, then for every $t,t_{0}>0$ 
\[
Y_{t}^{t_{0}}=Y_{t+t_{0}}-Y_{t_{0}}\overset{d}{\rightarrow}Y_{t}\qquad\alpha\rightarrow1.
\]
\end{prop}
\begin{proof}
In \cite[eq. 3.1.19]{slater1960confluent} it was shown that 
\begin{equation}
U\left(a,b,s\right)=\frac{1}{\Gamma\left(a\right)}\integral 0{\infty}e^{-sx}x^{a-1}\left(1+x\right)^{b-a-1}dx,\label{eq:U integral form}
\end{equation}
where $U\left(a,b,s\right)$ is a hypergeometric function that solves
the confluent hypergeometric equation, also known as Kummer's equation
\begin{equation}
s\frac{\partial^{2}U}{\partial^{2}s}+\left(b-s\right)\frac{\partial U}{\partial s}-aU=0.\label{eq:Kummer's equation}
\end{equation}
By (\ref{eq:U integral form}) and a simple change of variables we
find that the Laplace transform of the generalized Beta prime distribution
is given by 
\begin{equation}
\hat{p}_{t_{0}}\left(s\right)=\frac{U\left(1-\alpha,1-\alpha,st_{0}\right)}{\Gamma\left(\alpha\right)}\qquad s>0.\label{eq:LT of the GB}
\end{equation}
Using the identity $U\left(1-\alpha,1-\alpha,x\right)=e^{x}\Gamma\left(\alpha,x\right)$
where $\Gamma\left(\alpha,x\right)$ is the incomplete gamma function
defined by $\Gamma\left(\alpha,x\right)=\integral x{\infty}t^{\alpha-1}e^{-t}dt$
, we can write (\ref{eq:LT of the GB}) in a more familiar notation
\[
\hat{p}_{t_{0}}\left(s\right)=\frac{e^{st_{0}}\Gamma\left(\alpha,st_{0}\right)}{\Gamma\left(\alpha\right)}.
\]
Now, by dominated convergence 
\begin{alignat*}{1}
\underset{\alpha\rightarrow1}{\lim}\hat{p}_{t_{0}}\left(s\right) & =\underset{\alpha\rightarrow1}{\lim}\frac{e^{st_{0}}\integral{st_{0}}{\infty}r^{\alpha-1}e^{-r}dr}{\Gamma\left(\alpha\right)}\\
 & =\frac{e^{st_{0}}e^{-st_{0}}}{1}=1.
\end{alignat*}
Therefore, by \cite[Theorem 4.3]{kallenberg2002foundations} we have
$p_{t_{0}}\overset{w}{\rightarrow}\delta$ as $\alpha\rightarrow1$
where $\overset{w}{\rightarrow}$ denotes weak convergence of probability
measures and $\delta$ is the Dirac delta measure. For a Borel set
$B$ such that $0\notin B$ define 
\[
f\left(r\right)=\left\{ \begin{array}{cc}
P\left(Y_{t-r}\in B\right) & 0\leq r\leq t\\
0 & t<r
\end{array}\right.,
\]
 and note that %
$P\left(Y_{0}\in dx\right)=\delta_{0}\left(dx\right)$ and therefore
$P\left(Y_{t-r}\in B\right)=0$ at $r=t$. Consequently, Lemma \ref{lem: continuity of probability}
suggests that $f\left(r\right)$ is continuous. By the fact that 
\begin{alignat*}{1}
P\left(Y_{t}^{t_{0}}\in B\right) & =\integral 0{\infty}f\left(r\right)p_{t_{0}}\left(r\right)dr\\
 & =\integral 0tP\left(Y_{t-r}\in B\right)p_{t_{0}}\left(r\right)dr\rightarrow P\left(Y_{t}\in B\right),
\end{alignat*}
we also have $P\left(Y_{t}^{t_{0}}=0\right)\rightarrow P\left(Y_{t}=0\right)$
and the proof is complete.\end{proof}
\begin{rem}
It was shown in \cite[eq. 4.1.12]{slater1960confluent} that $U\left(a,b,s\right)\sim Cs^{-a}$
as $s\rightarrow\infty$. It follows that
\begin{equation}
\hat{p}_{t_{0}}\left(s\right)\sim C\left(st_{0}\right)^{\alpha-1}\qquad t_{0}\rightarrow\infty\label{eq:Asymtotics of LT of GB}
\end{equation}
and therefore $\hat{p}_{t_{0}}\left(s\right)\rightarrow0$ as $t_{0}\rightarrow\infty$.
Hence, $p_{t_{0}}\overset{v}{\rightarrow}0$ as $t_{0}\rightarrow\infty$
where $\overset{v}{\rightarrow}$ denotes vague convergence of distributions,
and $P\left(Y_{t}^{t_{0}}\in dx\right)\overset{w}{\rightarrow}\delta_{0}\left(dx\right)$,
another proof for the fact that $P\left(Y_{t}^{t_{0}}\in B\right)\rightarrow0$
as $t_{0}\rightarrow\infty$ for $B$ such that $0\notin B$. It is
not hard to verify that $\hat{p}_{t_{0}}\rightarrow0$ as $\alpha\rightarrow0$.
Intuitively, this is expected since a small $\alpha$ suggests long
waiting times between jumps and that $Y_{t}$ is very subdiffusive. 
\end{rem}
Let $p\left(dx,t\right)$ be a stochastic kernel, that is, for every
$t>0$ $p\left(dx,t\right)$ is a probability measure on $\sigma\left(\mathbb{R}\right)$
and for each Borel set $B\subseteq\mathbb{R}$ $p\left(B,\cdot\right)$
is measureable. Denote the Fourier transform of $p\left(dx,t\right)$
by $\widetilde{p}\left(k,t\right)=\intop_{\mathbb{R}}e^{-ikx}p\left(dx,t\right)$,
and the Fourier-Laplace transform (FLT) by $\overline{p}\left(k,s\right)=\intop_{\mathbb{R}^{+}}\intop_{\mathbb{R}}e^{-st-ikx}p\left(dx,t\right)dt$.
{} Recall the definition of the Caputo $0<\alpha<1$ fractional derivative
of a function $f\left(t\right)$,
\begin{equation}
\partial_{t}^{\alpha}f=\frac{1}{\Gamma\left(1-\alpha\right)}\integral 0t\left(t-r\right)^{-\alpha}\frac{\partial f\left(r\right)}{\partial r}dr.\label{eq:Caputo derivative}
\end{equation}
 For $0<\alpha<1$ the Laplace transform of $\partial_{t}^{\alpha}f$
is (\cite[p. 39]{Meerschaert2011a}) 
\begin{alignat*}{1}
\widehat{\partial_{t}^{\alpha}f} & =s^{\alpha}\hat{f}-s^{\alpha-1}f\left(0+\right).
\end{alignat*}
A closely related operator is the Riemann Liouville derivative $\mathbb{D}_{t}^{\alpha}$
for $0<\alpha<1$, which is defined by
\begin{alignat}{1}
\mathbb{D}_{t}^{\alpha}f & =\frac{1}{\Gamma\left(1-\alpha\right)}\frac{\partial}{\partial t}\integral 0t\left(t-r\right)^{-\alpha}f\left(r\right)dr.\label{eq:Riemann Liouville drivative}
\end{alignat}
The LT of (\ref{eq:Riemann Liouville drivative}) can be shown to
be $\widehat{\mathbb{D}_{t}^{\alpha}f}=s^{\alpha}\hat{f}$. It follows
that 
\begin{equation}
\partial_{t}^{\alpha}f=\mathbb{D}_{t}^{\alpha}f-f\left(0+\right)\frac{t^{-\alpha}}{\Gamma\left(1-\alpha\right)}.\label{eq:relation between RL and Caputo}
\end{equation}
The following is a short summary of results in \cite{baeumer2005space}.
Let $V^{\omega}=L_{\omega}^{1}\left(\mathbb{R}\times\mathbb{R}_{+}\right)$
be the space of real valued measurable functions on $\mathbb{R}\times\mathbb{R}_{+}$
such that 
\[
\left\Vert f\right\Vert _{\omega}=\integral 0{\infty}\integral{\mathbb{R}^{d}}{}e^{-\omega t}\left|f\left(x,t\right)\right|dxdt<\infty,
\]
for some $\omega>0$. $V^{\omega}$ is a Banach space w.r to $\left\Vert \cdot\right\Vert _{\omega}$.
If $\left(A_{t},D_{t}\right)$ is a L�vy process where $D_{t}$ is
a subordinator and s.t $E\left(e^{-ikA_{t}-sD_{t}}\right)=e^{t\eta\left(-k,s\right)}$,
then the distribution of $\left(A_{t},D_{t}\right)$ gives way to
a semi group of operators whose infinitesimal generator $L'$ satisfies
$\overline{L'f}=\eta\left(-k,s\right)\overline{f}\left(k,s\right)$(for
$\omega\leq s$). In fact, $f$ is in the domain of $L'$, $D\left(L'\right)$,
iff $\overline{g}\left(k,s\right)=\eta\left(-k,s\right)\overline{f}\left(k,s\right)$
where $\overline{g}\left(k,s\right)$ is the FLT of some $g\in V^{\omega}$.
If the first and second order spatial weak derivatives as well as
the first order time weak derivative of $f$ is in $V^{\omega}$ then
$f\in D\left(L'\right)$. Let $p^{t_{0}}\left(dx,t\right)$ be the
probability measure of the process $Y_{t}^{t_{0}}$, i.e. $p^{t_{0}}\left(dx,t\right)=P\left(Y_{t}^{t_{0}}\in dx\right)$.
Suppose $A_{t}$ has the symbol $\psi\left(k\right)$ and the infinitesimal
generator $L$, and $D_{t}$ is an independent standard stable subordinator.
We then have $\eta\left(-k,s\right)=-s^{\alpha}+\psi\left(-k\right)$,
and $L'=-\mathbb{D}_{t}^{\alpha}+L$ (since $f\in V^{\omega}$, $Lf$
should be understood as $f\left(\cdot,t\right)\in D\left(A\right)$
for every $t>0$). Note that by \cite[Theorem 2.2]{Baeumer2001} smooth
functions on $\mathbb{R}$ are contained in $D\left(L\right)$. The
FLT of $p^{0}\left(dx,t\right)$ is well known(\cite[Eq. 4.43]{Meerschaert2011a})
and given by 
\begin{equation}
\overline{p^{0}}\left(k,s\right)=\frac{s^{\alpha-1}}{-\eta\left(-k,s\right)}=\frac{s^{\alpha-1}}{s^{\alpha}-\psi\left(-k\right)},\label{eq:FLT of CTRWL}
\end{equation}
which in turn implies that
\begin{alignat}{1}
\partial_{t}^{\alpha}p^{0}\left(dx,t\right) & =Lp^{0}\left(dx,t\right)\label{eq:Fokker Planck equation}\\
p^{0}\left(dx,0\right) & =\delta_{0}\left(dx\right).\nonumber 
\end{alignat}
 Equation (\ref{eq:Fokker Planck equation}) describes the dynamics
of $p^{0}\left(dx,t\right)$ and therefore is called the Fractional
Fokker Planck Equation(FFPE) of $p^{0}\left(dx,t\right)$. Suppose
that the process $Y_{t}^{t_{0}}$ starts from the random point $X_{0}$
with density $p\left(x\right)\in C_{c}^{\infty}\left(\mathbb{R}\right)$,
that is, smooth with compact support and that $X_{0}$ is independent
of $Y_{t}^{t_{0}}.$ The distribution of $Y_{t}^{t_{0}}+X_{0}$ is
$C\left(x,t\right)=\integral{\mathbb{R}}{}p\left(x-y\right)p^{t_{0}}\left(dy,t\right)$
which is again smooth.  The next theorem obtains the governing equation
of $C\left(x,t\right)$.
\begin{thm}
\label{thm:Fokker Planck Equation}Let $Y_{t}=A_{E_{t}}$ have probability
measure $p^{0}\left(dx,t\right)$ whose FLT is given by (\ref{eq:FLT of CTRWL})
for $0<\alpha<1$. Let $L$ be the generator of $A_{t}$\textup{.
Then we have
\begin{alignat}{1}
\partial_{t}^{\alpha}C\left(x,t\right) & =L\left(C\left(x,t\right)-p\left(x\right)\integral t{\infty}p_{t_{0}}\left(r\right)dr\right)\label{eq:FPDE of C(x,t)}\\
C\left(x,0\right) & =p\left(x\right).\nonumber 
\end{alignat}
}\end{thm}
\begin{proof}
Let 
\begin{equation}
p\left(x,t\right)=C\left(x,t\right)-p\left(x\right)\integral t{\infty}p_{t_{0}}\left(r\right)dr,\label{eq:p(t,x)}
\end{equation}
and note that the FLT of (\ref{eq:p(t,x)}) is 
\begin{equation}
\overline{p}\left(k,s\right)=\overline{C}\left(k,s\right)-\widetilde{p}\left(k\right)\left(\frac{1}{s}-\frac{1}{s}\hat{p}_{t_{0}}\left(s\right)\right).\label{eq:FLT of p(t,x)}
\end{equation}
By Remark \ref{remark:zero increment probability} we have 
\begin{equation}
p\left(x,t\right)=\integral{\mathbb{R}}{}p\left(x-y\right)\integral 0tp^{0}\left(dy,t-r\right)p_{t_{0}}\left(r\right)dr.\label{eq: unnormalized p(x,t)}
\end{equation}
By a general version of Fubini's Theorem \cite[Theorem 2.6.4]{ash1972real}
we have 
\begin{alignat}{1}
\intop_{\mathbb{R}}e^{-ikx}\integral 0tp^{0}\left(dx,t-r\right)p_{t_{0}}\left(r\right)dr & =\integral 0t\widetilde{p}^{0}\left(k,t-r\right)p_{t_{0}}\left(r\right)dr.\label{eq:FT of p t1(t,x)}
\end{alignat}
Take the LT of both sides of equation (\ref{eq:FT of p t1(t,x)})
to obtain 
\[
\intop_{\mathbb{R}^{+}}\intop_{\mathbb{R}}e^{-st-ikx}\integral 0tp^{0}\left(dx,t-r\right)p_{t_{0}}\left(r\right)dr=\overline{p}^{0}\left(k,s\right)\widehat{p}_{t_{0}}\left(s\right).
\]
It follows that 
\[
\overline{p}\left(k,s\right)=\overline{p}^{0}\left(k,s\right)\widehat{p}_{t_{0}}\left(s\right)\widetilde{p}\left(k\right)
\]
Since by (\ref{eq:FLT of CTRWL}) $\overline{p}^{0}\left(k,s\right)=\frac{s^{\alpha-1}}{s^{\alpha}-\psi\left(-k\right)}$
we have
\begin{equation}
\overline{p}\left(k,s\right)=\frac{s^{\alpha-1}}{s^{\alpha}-\psi\left(-k\right)}\widehat{p}_{t_{0}}\left(s\right)\widetilde{p}\left(k\right).\label{eq:FLT eq}
\end{equation}
Substitute (\ref{eq:FLT of p(t,x)}) in (\ref{eq:FLT eq}) to obtain
\[
\overline{C}\left(k,s\right)s^{\alpha}-\overline{C}\left(k,s\right)\psi\left(-k\right)-\widetilde{p}\left(k\right)\left(\frac{1}{s}-\frac{1}{s}\widehat{p}_{t_{0}}\right)s^{\alpha}+\widetilde{p}\left(k\right)\left(\frac{1}{s}-\frac{1}{s}\widehat{p}_{t_{0}}\right)\psi\left(-k\right)=s^{\alpha-1}\widehat{p}_{t_{0}}\left(s\right)\widetilde{p}\left(k\right),
\]
which can be rearranged to obtain
\begin{equation}
\overline{C}\left(k,s\right)\left(s^{\alpha}-\psi\left(-k\right)\right)=\widetilde{p}\left(k\right)s^{\alpha-1}-\left(\frac{1}{s}-\frac{1}{s}\widehat{p}_{t_{0}}\right)\widetilde{p}\left(k\right)\psi\left(-k\right).\label{eq:FLT eq-1}
\end{equation}
By the preceding discussion the right hand side of \ref{eq:FLT eq-1}
inverts to a function in $V^{\omega}$, taking the IFLT of (\ref{eq:FLT eq-1})
we have 
\begin{equation}
\mathbb{D}_{t}^{\alpha}C\left(x,t\right)-LC\left(x,t\right)=p\left(x\right)\frac{t^{-\alpha}}{\Gamma\left(1-\alpha\right)}-Lp\left(x\right)\integral t{\infty}p_{t_{0}}\left(r\right)dr.\label{eq:FLT eq-2}
\end{equation}
Noting that $C\left(x,0^{+}\right)=p\left(x\right)$ one can rewrite
(\ref{eq:FLT eq-2}) by using (\ref{eq:relation between RL and Caputo})
to arrive at (\ref{eq:FPDE of C(x,t)}). \end{proof}
\begin{rem}
Although Equation (\ref{eq:FPDE of C(x,t)}) is not an abstract Cauchy
problem, one may adopt the concept of a mild solution from \cite[Chapter 4]{pazy2012semigroups}
and use it in our case. Let $f\in L^{1}\left(\mathbb{R}\right)$,
we say that a function $C\left(x,t\right)=f*p^{t_{0}}\left(dx,t\right)=\integral{\mathbb{R}}{}f\left(x-y\right)p^{t_{0}}\left(dy,t\right)$
is a mild solution of 
\begin{alignat}{1}
\partial_{t}^{\alpha}C\left(x,t\right) & =L\left(C\left(x,t\right)-f\left(x\right)\integral t{\infty}p_{t_{0}}\left(r\right)dr\right)\label{eq:mild solution1}\\
C\left(x,0\right) & =f\left(x\right)\nonumber 
\end{alignat}
 if there exists a sequence $\phi_{n}\in D\left(L\right)$ s.t $\phi_{n}\overset{L^{1}}{\rightarrow}f$
(this implies that $\phi_{n}*p^{t_{0}}\left(dx,t\right)\overset{L^{1}}{\rightarrow}f*p^{t_{0}}\left(dx,t\right)$
uniformly in $t$ on bounded sets as $p^{t_{0}}\left(dx,t\right)$
is a contraction for every $t$). From Theorem \ref{thm:Fokker Planck Equation}
and the fact that $C_{c}^{\infty}\left(\mathbb{R}\right)$ is dense
in $L^{1}\left(\mathbb{R}\right)$, we conclude that every $f\in L^{1}\left(\mathbb{R}\right)$
is a mild solution of (\ref{eq:mild solution1}). We then write for
simplicity
\begin{alignat}{1}
\partial_{t}^{\alpha}p^{t_{0}}\left(dx,t\right) & =L\left(p^{t_{0}}\left(dx,t\right)-\delta_{0}\left(dx\right)\integral t{\infty}p_{t_{0}}\left(r\right)dr\right)\label{eq:mild solution2}\\
p^{t_{0}}\left(dx,0\right) & =\delta_{0}\left(dx\right).\nonumber 
\end{alignat}
Since (\ref{eq:mild solution2}) describes the dynamics of the probability
kernel $p^{t_{0}}\left(dx,t\right)$ we call it its FFPE.
\end{rem}

\begin{rem}
Theorem \ref{thm:Fokker Planck Equation} shows that the dynamics
of $p^{t_{0}}\left(dx,t\right)$ are the same as those of $p^{0}\left(dx,t\right)$
on $\mathbb{R}/\{0\}\times[0,\infty)$. There is a nice intuitive
interpretation to equation (\ref{eq:mild solution2}) when $A_{t}$
is a stable process. Equation (\ref{eq:FPDE of C(x,t)}) can be explained
as the behavior of a plume of particles by arguments of conservation
of mass and Fick's law (\cite[Remark 2.3]{Meerschaert2011a} and \cite[Section 16.1]{zaslavsky2008hamiltonian}).
However, note that the portion of the mass of particles that does
not diffuse away from point $x=0$ at time $t$ (and therefore does
not contribute to the change in $p^{t_{0}}\left(dx,t\right)$ over
time) is $\integral t{\infty}p_{t_{0}}\left(r\right)dr$ by Remark
\ref{remark:zero increment probability} and the fact that stable
processes have pdf. This accounts for the difference between (\ref{eq:mild solution2})
and (\ref{eq:Fokker Planck equation}).
\end{rem}

\begin{rem}
In \cite{Barkai2003a}, a deterministic system was modeled by a CTRW
and its aging properties were studied. There, the FFPE was given for
the unnormalized distribution (\ref{eq:p(t,x)}) of the aging process
when $t_{0}$ and $t$ are large. To see that the results agree, simply
plug $\psi\left(-k\right)=\frac{-k^{2}}{2A}$ in (\ref{eq:FLT eq})
and take the IFLT of both sides of the equation. Since $\overline{p}\left(k,s\right)s^{\alpha}$
is the FLT of the fractional Riemann-Liouville $\alpha$ derivative
we obtain \cite[eq. 18]{Barkai2003a}.
\end{rem}

\bibliographystyle{abbrv}
\bibliography{ACTRWL}

\end{document}